\documentclass{article}
\usepackage{hyperref, url}
\usepackage{amscd,amsmath,amsthm,amssymb}
\usepackage{enumerate,varioref}
\usepackage{epsfig}
\usepackage{graphicx}
\usepackage{subfig}
\usepackage{mathtools}
\usepackage{tikz}
\usepackage{algorithmicx}
\usepackage{algorithm}
\usepackage{algpseudocode}
\usepackage{xcolor}
\newtheorem{thm}{Theorem}

\newtheorem{lem}[thm]{Lemma}

\theoremstyle{definition}
\newtheorem{defn}[thm]{Definition}
\theoremstyle{remark}
\newtheorem{ex}[thm]{Example}

\newcommand{\R}{\mathbb{R}}

\title{An Optimal Constraint for QUBO Models}
\author{Clark Alexander\footnote{Chicago Quantum} \\ email: \href{mailto:mailto:research@quantum-usaci.com}{the author}}

\begin{document}
	
	\maketitle
	
	\begin{abstract}
		A quadratic binary unconstrained optimization model, hereafter QUBO, by definition is unconstrained.  This, however, is not ideal if one needs to select a model containing only a fixed size binary vector.  In this work we show how to add a constraint to a QUBO to force a particular size solution.
	\end{abstract}

    \section{Introduction}
    
    As the industry of quantum computing has been making inroads in applications, the first set of problems which are in-line for fast, efficient ``solutions" are combinatorial optimization problems \cite{CA, DW1, STB}
    As of this writing, the technique which appears to be furthest along is quantum annealing.  This is an appealing technique as it allows one with a reasonably powerful laptop and a high level language compiler to benchmark the quantum annealing process.  Generally speaking, the most-straightforward way to apply quantum annealing is to pose one's problem in terms of a quadratic form.  In particular, if one can pose a problem in terms of a solution which marks a set of elements as ``in" or ``out" then quantum annealing works well.  Additionally, there are many classical, heuristic, and probabilistic techniques which one can use to benchmark a quantum annealer; for example, simulated annealing, genetic algorithms, simulated bifurcation machines, branch-and-bound, etc. \cite{DigiAnn, GKD, GTD, EKU}.  
    
    However, quadratic forms with ``in"/``out" solutions require binary vectors and depending on the particular type of quadratic form given, one can expect a very specific set of solutions to show up.  For example, if one provides a positive definite matrix as a quadratic form then the known minimum is the zero vector.  There is no need to apply probabilistic techniques, all non zero vectors will give positive solutions and in searching for a minimum, we cannot go lower than zero. Additionally if the quadratic form is a well known random matrix, a Gaussian ensemble, for instance, then the best solution will contain roughly half ``in" and half ``out."  These solutions are expected, but maybe not completely useful.  
    
    Recently, there has been a push for using quadratic forms in financial optimization \cite{CKA40, CKA60, MLO}  One common strategy is to maximize expected return while minimizing volatility.  Volatility is, however, covariance, which is positive semi-definite.  In reality covariance is positive definite with real data.  Thus trying to minimize a positive definite matrix reveals the ``mathematically correct"  solution of not investing as that reduces volatility to zero.  However, for an investor, this is not a useful solution. Perhaps an investor wants to invest in an ``optimal" set of 15 assets.  An unconstrained problem will never produce such an answer, thus one wishes to add a constraint in a somewhat natural way so that a set of 15 assets becomes visible.

    \section{Elementary Properties of a Random QUBO}
    
    A quadratic unconstrained binary optimization (hereafter QUBO) can be thought of as a square matrix.  While this is a slight abuse of notation, there should not be too much confusion as the actual optimization involves finding a vector for the particular matrix.  So where it is clear we will refer to the matrix and the optimization problem both as QUBO.  
    \begin{defn}
    	Given a real symmetric matrix $A\in \R^{N\times N}$ and a real vector $B\in \R^N$, a quadratic unconstrained binary optimization is a problem in which one seeks the vector $x \in  \{0,1\}^N$ so that 
    	\begin{equation}
    	Q(A,B) = x^t A x + B\cdot x
    	\end{equation}
    	achieves a minimum value.
    	That is
    	\begin{equation}
    	\min_x x^t A x + B\cdot x
    	\end{equation}
    \end{defn}
    
    It is also common to expect $A$ to be upper or lower triangular and traceless.  However by considering $(A+A^t)/2$ one can trivially move back and forth between triangular and symmetric matrices.  We prefer the symmetric matrix as it guarantees real eigenvalues and is more in-line with the principles of quantum mechanics and therefore allows one to pose a matrix $A$ (or a small variant thereof) as a Hamiltonian in an Ising model which is the current technology used in quantum annealing and simulated bifurcation machines.  This also allows one the ability to easily translate between quantum, digital, and simulated annealing for bench marking purposes.

    \begin{lem}
    	While considering a binary vector $x\in\{0,1\}^N$ one can reduce a QUBO into a single matrix (which need not be traceless).  That is one can reduce
    	\[
    	x^t A x + B\cdot x \rightarrow x^t \tilde{A} x
    	\]
    	
    	where 
    	\begin{equation}
    	\tilde{A} = A + \text{diag}(B)
    	\end{equation}
    	with diag$(B)$ being a diagonal matrix with 
    	\[ 
    	diag(B)_{ij} = B_i \delta_{ij} 
    	\]
    \end{lem}	
    
    \begin{proof}
    	Since each $x_i$ is 0 or 1 we trivially have $x_i = x_i^2$.
    	This reveals
    	\[
    	B\cdot x = x^t \text{diag}(B) x
    	\]
    	
    	Which allows us to factor $x^t$ on the left and $x$ on the right
    	\begin{eqnarray}
    	x^tAx + B\cdot x & = & x^t A x + x^t \text{diag}(B) x \nonumber\\
    	                 & = & x^t (A + \text{diag}(B)) x \nonumber\\
    	                 & = & x^t \tilde{A} x
    	\end{eqnarray}
    	
    \end{proof}

    From here forward we shall simply refer to $\tilde{A}$ as $A$.  
    
    In quantum annealing, the vector $x_i \in \{0,1\}$ is exchanged for a vector $z_i \in \{-1,1\}$ with the simple transformation
    \[
    z_i = 2x_i - 1
    \]
    This affects our matrix $A$ and also gives one an offset vector where
    \begin{eqnarray}
    x^t A x = z^t J z + C\cdot z + \text{const}
    \end{eqnarray}
    
    The matrix $J = A/4$ and offset vector $C$ are calculated by a simple change of variables.  For the purposes of this article we will remain in the space $x_i \in \{0,1\}$ and only mention that can transform when necessary.

    \section{The Constraint Matrix}
    
    Given our model
    \[
    \min_x x^t A x
    \]
    with $A$ a real symmetric matrix of size $N\times N$ we wish to add a constraint matrix $C$ to $A$ so that our vector $x$ has norm $\|x\|^2 = m$; or more simply the $L^1$ norm is $M$, $\|x\|_1 = M$ for some $1 \le M \le N$
    
    Thus our new model becomes
    \[
    \min_x x^t (A+C) x
    \] 
   
   If we pass this quadratic form to a solver (whether a simulated annealer or quantum annealing computer or a simulated bifurcation machine) we can expect that a ``good" solver will produce a vector $x$ of the required size.
   
   \begin{thm}
   	Let $A$ be a real symmetric matrix of size $N\times N$.  Then the addition of matrix
   	\begin{equation}
   	C = \alpha J_N + \beta I_N
   	\end{equation}
   	to $A$ will produce a QUBO with optimal vector $x$ of size $\|x\|_1 = M$
   	when 
   	\begin{equation}
   	M = \frac{-\beta}{2\alpha} \label{optimal parameters}
   	\end{equation}
   	In particular there is a line of solutions.  The larger $\alpha$ the stronger the pull toward $\|x\|_1 = M$. 
   	
   \end{thm}
   
   \begin{proof}
   	 Since we wish to minimize $x^t A x$ where $|x|_1 = M$ we consider the addition of the matrix $C = \alpha J_N + \beta I_N$ where $J_N$ is the $N\times N$ matrix of all ones and $I_N$ is the $N \times N$ identity matrix.
   	 
   	 This reduces our calculation to 
   	 \begin{equation}
   	 \min_x x^t(A + C)x \implies \min_{|x|_1 = M} x^t C x 
   	 \end{equation}
   	 
   	 Since $J_N$ and $I_N$ yield straightforward multiplications we have
   	 \begin{equation}
   	 x^t(\alpha J_N + \beta I_N) x = \alpha M^2 + \beta M
   	 \end{equation}
   	 
   	 Now we see the results directly.  We choose $\alpha$ and $\beta$ to minimize $x^tCx$ by
   	 \begin{equation}
   	 M = \frac{-\beta}{2\alpha}
   	 \end{equation}
   	 
   	 Noting that $\alpha > 0$ gives a minimum and $\alpha < 0$ gives a maximum.
   	 
   	 Further more we can reparameterize this to be a single parameter constraint as
   	 \begin{equation}
   	 C(\alpha) = \alpha(J_N - 2MI_N)
   	 \end{equation}
   	 We can also see this as
   	 \[
   	 C(\alpha) = \alpha\begin{bmatrix} 
   	 1-2M & 1 & 1 & \dots & 1\\
      1 & 1-2M & 1 & \dots & 1 \\
      1 & 1 & 1-2M & \dots & 1\\
       \vdots & \vdots & \vdots & \ddots & \vdots\\
      1 & 1 & 1 & \dots & 1-2M
    \end{bmatrix}   	 \]
   	 
   \end{proof}

\section{Two Examples}
\begin{ex}
	
Let's take a quick look at how to simulate this in a modern computing language.  We'll produce a random symmetric matrix of size $N\times N$ and require a solution of size $M$ where $M$ is significantly different from $N/2$.

In the code on the author's local computer, we will run a simulated annealer which produces two outputs, (a) the cost $x^tAx$, (b) the solution vector.

\begin{algorithm}
	\caption{Picking M asets; Julia/Octave style}
	\label{Adjacency to rotation}
	\begin{algorithmic}
		\State $A = randn(N,N)$
		\State $A = (A+ A^t)/2$
		\State $C(\alpha) = \alpha*(\text{ones}(N,N) - 2M* \text{eye}(N))$
		\State cost, solution = anneal($A+C(0.5)$)
		\State sum(solution) 
	\end{algorithmic}	
\end{algorithm}
    
Using five hundred runs with parameters $N=30, M=8$ we have a histograms 
for $\alpha \in  \{0, 0.1, 0.2, 0.5, 1, 2, 10\}$. When $\alpha = 0$ we have the unconstrained model, where ones expects a solution of size $15 \pm \sqrt{30}$ so we expect between 10 and 20 as the global best solution.  We are running our simulated annealer cooling very quickly and only a few trials per degree so as to show the efficacy of the constraint.  We can see the pull toward 8 assets as $\alpha$ increases.

\begin{figure}[!ht]
	\centering
	\subfloat[Unconstrained]{{\includegraphics[width=5.5cm]{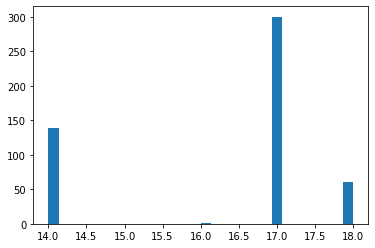} }}%
	\qquad
	\subfloat[$\alpha = .1$]{{\includegraphics[width=5.5cm]{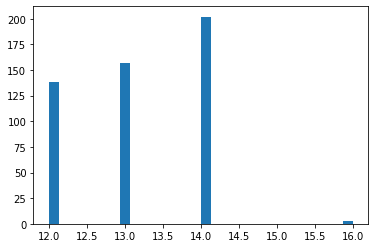} }}%
	\newline
	\subfloat[$\alpha = 0.2$]{{\includegraphics[width=5.5cm]{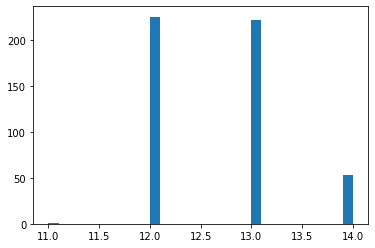} }}%
	\qquad
	\subfloat[$\alpha = 0.5$]{{\includegraphics[width=5.5cm]{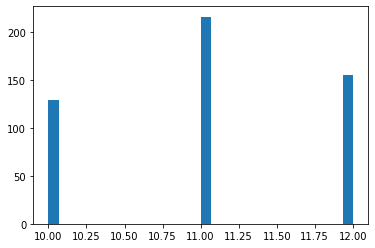} }}%
	\caption{Histograms with increasing constraints}
	\label{fig: Gaussian charts 1}
\end{figure}
	
\begin{figure}[!ht]	
	\qquad
	\subfloat[$\alpha = 1$]{{\includegraphics[width=5.5cm]{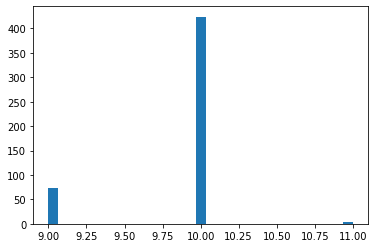} }}%
	\qquad
	\subfloat[$\alpha = 2$]{{\includegraphics[width=5.5cm]{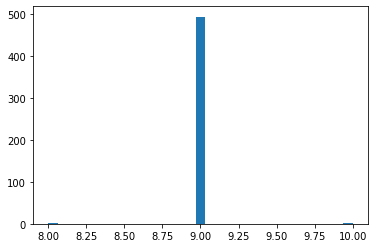} }}%
	\qquad
	\subfloat[$\alpha = 10$]{{\includegraphics[width=5.5cm]{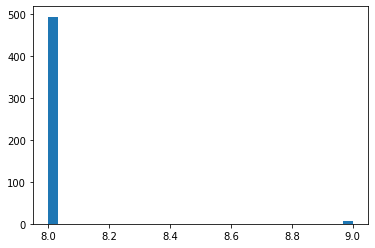} }}%
	\qquad
	\subfloat[$\alpha = 0.5$]{{\includegraphics[width=5.5cm]{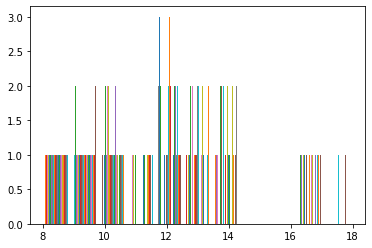} }}%
	\caption{Histograms with increasing constraints
	}
	\label{fig:Gaussian charts 2}
\end{figure}

The final frame in \ref{fig:Gaussian charts 2} is a stacked histogram showing all 7 histograms together.  We see the density of selections increasing as we move toward 8.  Looking all the way back to the first figure in \ref{fig: Gaussian charts 1} we see that the global minimum is likely larger than 15, most likely at 17.  If we were to have a larger Gaussian ensemble or one in which the global minimum is at or below $N/2$ then the shift toward a smaller number of selected assets is even clearer.  In this case, if we had selected the number of assets to be 20, we would see the rightward shift in the histograms more strongly.

\end{ex}

\begin{ex}

For our second example, we'll use a positive semi-definite matrix of size $30\times 30$ and again apply constraints to select 8 assets.

It is important to note that the simulated annealer on the author's computer tends to avoid picking exactly zero assets.  Thus in our first figure in \ref{fig: Positive Definite 1} we see a split between zero and one, even though we know the mathematically sound answer is exactly zero.

\begin{figure}[!ht]
	\centering
	\subfloat[Unconstrained]{{\includegraphics[width=5.5cm]{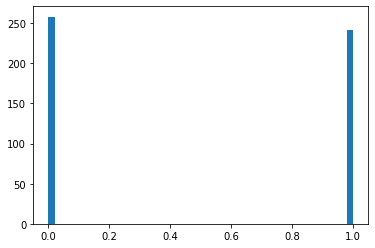} }}%
	\qquad
	\subfloat[$\alpha = .1$]{{\includegraphics[width=5.5cm]{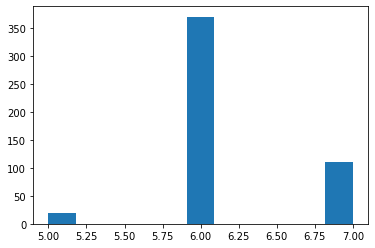} }}%
	\newline
	\subfloat[$\alpha = 0.2$]{{\includegraphics[width=5.5cm]{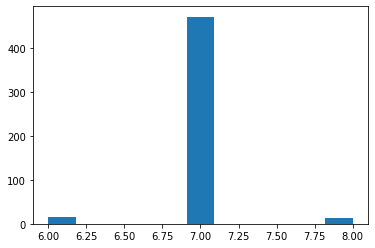} }}%
	\qquad
	\subfloat[$\alpha = 0.5$]{{\includegraphics[width=5.5cm]{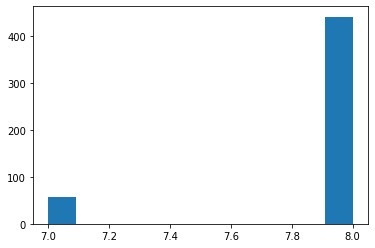} }}%
	\caption{Histograms with increasing constraints}
	\label{fig: Positive Definite 1}
\end{figure}

\begin{figure}[!ht]	
	\qquad
	\subfloat[$\alpha = 1$]{{\includegraphics[width=5.5cm]{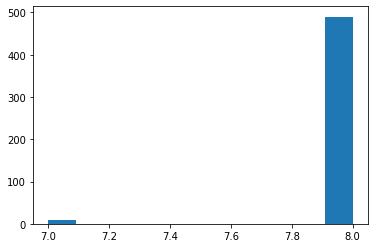} }}%
	\qquad
	\subfloat[$\alpha = 2$]{{\includegraphics[width=5.5cm]{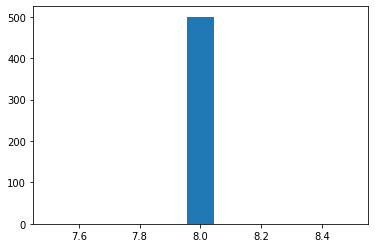} }}%
	\qquad
	\subfloat[$\alpha = 10$]{{\includegraphics[width=5.5cm]{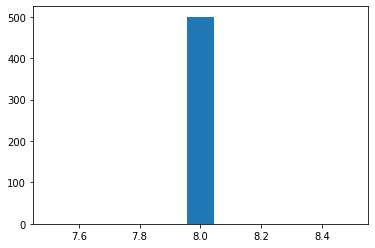} }}%
	\qquad
	\subfloat[$\alpha = 0.5$]{{\includegraphics[width=5.5cm]{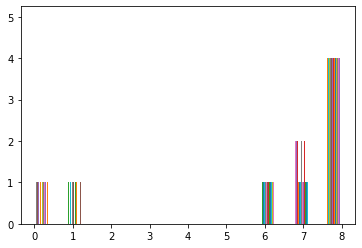} }}%
	\caption{Histograms with increasing constraints
	}
	\label{fig:Positive Definite 2}
\end{figure}

We see the efficacy of the constraint here much more prominently.  Simulated annealers are not necessary with such small scale QUBO as $30\times 30$ as a brute force solution can be obtained in a matter of minutes.  Additionally, small shifts in small QUBO models don't have such a large effect.  If one were to repeat these examples with $N = 5000$ and $M = 50$ the effects of $\alpha = 0.1$ would appear much more clearly.  Nonetheless even with $\alpha$ as small as $0.2$ we begin seeing solutions with 8 assets.  At $\alpha= 0.5$ the lion's share are solutions with 8 assets (442 out of 500 in this particular numerical experiment).  And at $\alpha = 2$ and above we select 8 assets without fail.  Thus is the last figure in \ref{fig:Positive Definite 2} we see an intense density at 8 assets.  With all other assets smaller.  This is consistent with the first example in which 8 assets becomes the minimum number of assets chosen as the unconstrained solution has greater than 8 assets, in this case the unconstrained solution has fewer.

\end{ex}

\end{document}